\documentclass{amsart}
\usepackage{graphicx}
\usepackage{latexsym}
\usepackage{hyperref}
\usepackage{amsfonts}
\usepackage[all]{xy}
\newtheorem{corollary}{Corollary}
\newtheorem{definition}{Definition}
\newtheorem{lemma}{Lemma}

\newtheorem{theorem}{Theorem}
\newtheorem{example}{Example}
\numberwithin{equation}{section}

\newcommand{\la}{\lambda}
\newcommand{\be}{\begin{eqnarray}}
\newcommand{\en}{\end{eqnarray}}
\newcommand{\no}{\nonumber}

\begin{document}
\title[]{On the $h$-Almost Ricci Soliton}
\author{J. N. Gomes$^1$, Qiaoling Wang$^2$}
\address{$^{1}$Departamento de Matem\'{a}tica-UFAM\\
69077-000-Manaus-AM, Brasil} \email{jnvgomes@pq.cnpq.br}
\thanks{$^{1}$Partially supported by PRONEX/CNPq/FAPEAM-BR}
\urladdr{http://www.ufam.edu.br}
\address{$^{2}$Departamento de Matem\'{a}tica-UnB\\
70910-900-Brasilia-DF, Brasil} \email{wang@mat.unb.br}
\thanks{$^{2}$Partially supported by CNPq-BR}
\urladdr{http://www.mat.unb.br}
\author{Changyu Xia$^3$}
\address{$^{3}$Departamento de Matem\'{a}tica-UnB\\
70910-900-Brasilia-DF, Brasil} \email{xia@mat.unb.br}
\thanks{$^{3}$Partially supported by CNPq-BR}
\urladdr{http://www.mat.unb.br}
\begin{abstract}
We introduce the concept {\it $h$-almost Ricci soliton} which extends naturally the {\it almost Ricci soliton} by Pigola-Rigoli-Rimoldi-Setti and show that a compact nontrivial $h$-almost Ricci soliton of dimension no less than three with $h$ having defined signal and constant scalar curvature is isometric to a standard sphere with the potential function well determined. We also consider the {\it $h$-Ricci soliton} which is a particular case of the $h$-almost Ricci soliton and a generalization of the {\it Ricci soliton} and give characterizations for a special class of gradient $h$-Ricci solitons.
\end{abstract}

\maketitle

\section{Introduction}

In the late 20th century Hamilton introduced the Ricci flow. More specifically, given a one-parameter family of  metrics $g(t)$ on a Riemannian manifold
$M^n$, defined on an interval $ I \subset \mathbb{R} $, denoting by $Ric_{g(t)}$ the Ricci tensor of the metric $g(t)$, the equation of Ricci flow is
\begin{equation}\label{eq1}
\frac{\partial}{\partial t}g(t)=-2Ric_{g(t)}.
\end{equation}

In \cite{hamilton1} Hamilton proved that for any smooth metric $g_0$ on a compact Riemannian manifold $M^n$, there exists a unique solution $g(t)$ to the equation (\ref{eq1}) defined on some interval $[0, \varepsilon)$, $\varepsilon> 0$, with $g(0)=g_0$. For the complete non-compact case, Wan-Xiong Shi proved in \cite{Shi} the existence of a complete solution of \eqref{eq1} under the condition that the sectional curvatures of $(M^n, g_0)$ are bounded.

A Ricci soliton  is a Ricci flow $(M^n, g(t))$, $0\leq t <T \leq +\infty,$ with the property that for each $ t\in [0, T)$, there exists a diffeomorphism
$\varphi_t: M^n\to M^n$ and a constant $\sigma(t)$ such that $\sigma(t)\varphi_t^*g_0= g(t)$. One way to generate Ricci solitons is as follows. Consider a  Riemannian manifold $(M^n, g_0)$ with a vector field $X$ and a constant $\lambda$, satisfying
\begin{equation}\label{eqfund1}
Ric_{g_0}+\frac{1}{2}\mathcal{L}_{X}g_0=\lambda g_0,
\end{equation}
where $\mathcal{L}_{X}g_0$ denotes the Lie derivative of $g_0$ with respect to $X$. Let us set $T:=\infty$, if $\lambda\leq0$, and
$T:=\frac{1}{2\lambda}$ if $\lambda>0$. Then, we define a function $\sigma (t)=-2\lambda t+1$, $t\in [0, T)$, and a vector field $Y\in{\mathfrak X}(M^n)$ by
$Y(x)=\frac{X(x)}{\sigma(t)}$,  $x\in M^n$, and finally, just let $\varphi_t$ be the one-parameter family of diffeomorphisms generated by $Y$. This characterization allows some authors to consider the equation \eqref{eqfund1} as the definition of Ricci soliton. For more details on  Ricci flow we refer the reader to \cite{Chow}.

Following the same line of defining Ricci soliton, it is natural to analyze the following situation: Let $(M^n, g_0)$ be a complete Riemannian manifold and
$g(t)$ be a solution of \eqref{eq1} defined on an interval $[0, \varepsilon)$, $\varepsilon >0$, such that $\varphi_t$ is a one-parameter family of  diffeomorphisms of $M^n$, with $\varphi_0 =id_M$ and $g(t)(x)=\tau(x,t)\varphi_t ^*g_0(x)$ for every $x\in M^n$, where $\tau(x,t)$ is a positive smooth  function on
$M^n\times [0,\varepsilon)$. In this case, we have
\begin{equation*}
\frac{\partial}{\partial t}g(t)(x)=\frac{\partial}{\partial t}\tau(x,t)\varphi_t^*g_0(x)+\tau(x,t)\varphi_t^*\mathcal{L}_{\frac{\partial}{\partial t}\varphi(x,t)}g_0(x).
\end{equation*}
Since $\varphi_0 =id_M$ and $g(0)=g_0$, we know that $\tau(x,0)=1$. It then follows from the above equation that
\begin{equation*}
Ric_{g_0}+ \frac{1}{2}\mathcal{L}_Xg_0=\lambda g_0,
\end{equation*}
where $\lambda(x)=-\frac{1}{2}\frac{\partial}{\partial t}\tau(x,0)$ and $X=\frac{\partial}{\partial t}\varphi(x,0),$ which motivates the definition of Almost Ricci Soliton given by the authors in \cite{br2,prrs}. However, we will see that some important geometric properties appear when we consider the  following definition.
\begin{definition}
An $h$-almost Ricci soliton is a complete Riemannian  manifold $(M^n,g)$ with a vector field $X\in\mathfrak{X}(M)$, a soliton function
$\lambda: M\to \Bbb{R}$ and a function $h:M\to\Bbb{R}$ which are smooth and satisfy the equation:
\begin{equation} \label{eqfund1h}
Ric_{g}+ \frac{h}{2}\mathcal{L}_Xg=\lambda g.
\end{equation}
\end{definition}
For convenience's sake we denote by $(M^n, g, X, h, \lambda)$ an $h$-almost Ricci soliton. In the case $\lambda$ is constant we simply say that it is an $h$-{\it Ricci soliton}.
When $\mathcal{L}_Xg=\mathcal{L}_{\nabla u}g$ for some smooth function $u:M\to\Bbb{R}$, we call $(M^n,g,\nabla u,h,\lambda)$ a gradient $h$-almost  Ricci soliton with \emph{potential function} $u$. In this case, the fundamental equation \eqref{eqfund1h} can be rewritten as
\begin{equation}\label{eqfund2h}
Ric+h\nabla^2 u=\lambda g,
\end{equation}
where $\nabla^2 u$ denotes the Hessian of $u$.

It should be mentioned that the latter case arises naturally in the warped product metrics. To see this, we use the same notation and terminology as in Barret O'Neill \cite{oneill}. So let us consider the warped product manifold $(B\times_f\Bbb{F},g)$, with $g=g_B+f^2\langle,\rangle$. One can check that it's Ricci tensor satisfies the following equations :
\begin{itemize}
\item [(i)] $Ric(X,Y)=Ric_B(X,Y)-\frac{m}{f}\nabla^2f(X,Y)$,
\item [(ii)] $Ric(X,V)=0$,
\item [(iii)] $Ric(V,W)=Ric_{\Bbb{F}}(V,W)- \Big(\frac{\Delta f}{f}+\frac{|\nabla f|^2}{f^2}(m-1)\Big)g(V,W)$,
\end{itemize}
for all horizontal vectors $X,Y$ and vertical vectors $V,W$, where $\nabla^2 f$ denotes the Hessian of a smooth function $f$ on the Riemannian manifold $(B,g_B)$, $m>1$ is the dimension of the Riemannian manifold $(\Bbb{F},\langle,\rangle)$. Thus, if the Ricci tensor satisfies  $Ric(X,Y)=\lambda g_B(X,Y)$, for all $X,Y\in\mathfrak{X}(B)$ and for some smooth function $\lambda$ on $B$, then we get a  gradient $(-\frac{m}{f})$-almost Ricci soliton $(B,g_B,\nabla f,-\frac{m}{f},\lambda)$. Also, if $B\times_f\Bbb{F}$ is Einstein, then the fiber $(\Bbb{F},\langle,\rangle)$ is automatically  Einstein. Conversely, if $(\Bbb{F},\langle,\rangle)$ is Einstein with $Ric_{\Bbb{F}}=\mu \langle,\rangle$, it then follows from (i), (ii) and (iii) that  the warped product manifold $B\times_{f}\Bbb{F}$ is Einstein with $Ric=\lambda g$ if and only if $(B,g_B)$ is a gradient $(-\frac{m}{f})$-Ricci soliton with potential function $f$ and soliton function $\lambda$ satisfying
\begin{equation}\label{RicWP}
\mu =f\Delta f+(m-1)\vert\nabla f\vert^{2}+\lambda f^{2}.
\end{equation}
This characterization is especially important as we shall see in Theorem \ref{thmConstEin}.  Due to the close connection with this fact, we recommend that the reader could also see the works in \cite{kim,case1}.

Let us say that an $h$-almost Ricci soliton is \emph{expanding}, \emph{steady} or \emph{shrinking} if $ \lambda$ is respectively negative, zero or positive on $M$ and  that it is \emph{undefined} if $\lambda$ has no definite sign. When $X$ is a homothetic conformal  vector field, that is,  $\mathcal{L}_Xg=c g$, for some constant $c$, $(M^n, g, X, h, \lambda)$  is said to be \emph{trivial}. Otherwise it is \emph{nontrivial}. Observe that the traditional Ricci soliton is  a $1$-Ricci soliton with constant $\lambda$.
Moreover, $1$-almost Ricci soliton is just the  almost Ricci soliton, whose geometry was first studied in \cite{prrs} where the authors  proved some existence results for  almost gradient Ricci solitons. Later, some  structural equations for
the  almost Ricci solitons were presented in \cite{br2} which resulted in several studies on the geometry of almost Ricci solitons (Cf. \cite{bgr,br2}).

In \cite{master} Maschler studied the equation \eqref{eqfund2h} free of our motivation and he referred to equation \eqref{eqfund2h} as Ricci-Hessian equation. Furthermore, we note that the Ricci-Hessian equation is related to a new class of Riemannian metrics introduced by Catino \cite{catino} which are natural generalizations of the Einstein metrics. More precisely, he called a Riemannaina manifold $(M^n, g) $ with $ n\geq 2$,  a generalized quasi-Einstein manifold, if there are smooth functions $f$, $\lambda$ and $\mu$ on $M$ satisfying
\begin{equation}\label{eqfund QEM}
Ric+\nabla^2f-\mu df\otimes df=\lambda g.
\end{equation}
When $\mu=\frac{1}{m}$, where $m$ is a positive integer, the above generalized quasi-Einstein manifold is called a generalized $m$-quasi-Einstein manifold (Cf. \cite{br3}) which will be denoted by  $\{M^n,g,\nabla f,\lambda,m\}$, and simply $m$-quasi-Einstein manifold when $\lambda$ is constant. It has been proved  in \cite{kim,case1} that $m$-quasi-Einstein manifolds are directly related to the warped product Einstein manifolds, and in \cite{case1} that any compact $m$-quasi-Einstein manifold with scalar curvature is trivial which means that $f$ is a constant. However, Barros and Ribeiro \cite{br3} presented a family of nontrivial generalized $m$-quasi-Einstein metrics on a Euclidean unit sphere $\mathbb{S}^n(1)$ that are rigid in the class of constant scalar curvature (see \cite{bjng}). Namely, they showed that a nontrivial compact generalized  $m$-quasi-Einstein manifold $\{M^n,g,\nabla f,\lambda,m\}$ with and constant scalar curvature is isometric to a standard Euclidean sphere $\Bbb{S}^{n}(r),$ and up to constant, $f=-m\ln(\tau-\frac{h_v}{n})$, where $\tau\in(\frac{1}{n},+\infty)$ is a real number and $h_v$ is the height function with respect to a fixed unit vector $v\in\mathbb{R}^{n+1}$, which is part of the family presented in \cite{br3}.

Returning to \eqref{eqfund QEM}, with $\mu=-\frac 1m$, we consider a nonconstant function $u=e^{\frac{f}{m}}.$ It is easy to see that
\begin{equation}\label{uandf}
\nabla u=\frac{1}{m}e^{\frac{f}{m}}\nabla f \ \ \ \mbox{and}\ \ \ \frac{m}{u}\nabla^2 u=\nabla^2f+\frac{1}{m}df\otimes df.
\end{equation}
Consequently, the equation \eqref{eqfund QEM} can be rewritten as
\begin{equation}\label{eqfund22}
Ric+\frac{m}{u}\nabla^2u=\lambda g.
\end{equation}
Therefore, every generalized quasi-Einstein manifold, with $\mu=-\frac{1}{m}$, is a gradient $\frac{m}{u}$-almost  Ricci soliton. In particular, interchanging $m$ by $-m$ in the calculations of \eqref{uandf}-\eqref{eqfund22} we conclude that all generalized $m$-quasi-Einstein metrics is a gradient $(-\frac{m}{u})$-almost Ricci soliton.

For what follows we assume that $M^n$ has dimension $n\geq3$ and say that $h$ has defined signal if either $h>0$ on $M$ or $h<0$ on $M$. In this paper, we will prove that $h$-almost Ricci solitons, with $h$ having defined signal, are rigid in the class of compact manifolds with constant scalar curvature. Namely, we have the following result.

\begin{theorem}\label{thmA}
A compact nontrivial $h$-almost Ricci soliton $(M^n, g, X, h, \la )$ of constant scalar curvature with $h$ having defined signal is isometric to a standard sphere ${\mathbb S}^n(r)$. Moreover, it is gradient  and the potential function is an eigenfunction corresponding to the  first eigenvalue of ${\mathbb S}^n(r)$.
\end{theorem}
The next result  gives strong restrictions to a special class of gradient $h$-Ricci solitons.
\begin{theorem}\label{thmConstEin}
Let $u$ be a positive function on $M^n$ and $m$ a nonzero constant.
\begin{itemize}
\item [(i)] For a gradient $(-\frac{m}{u})$-Ricci soliton $(M^n, g, \nabla u, -\frac{m}{u},\la)$, the following is valid
\begin{equation}\label{eqthmCE}
\lambda u^2 + u\Delta u + (m-1)|\nabla u|^2 = \mu,
\end{equation}
for some constant $\mu$.
\item [(ii)] A compact  steady or expanding  gradient $(-\frac{m}{u})$-Ricci soliton $(M^n, g, \nabla u, -\frac{m}{u},\la)$ is trivial.
\end{itemize}
\end{theorem}

By taking $m$ to be an integer no less than $2$ in Theorem \ref{thmConstEin} and using the  O'Neill's formulas above, we can construct warped product  Einstein metrics as follows.

\begin{corollary}
Let $(B,g_B,\nabla u,-\frac{m}{u},\lambda)$ be a  gradient $(-\frac{m}{u})$-Ricci soliton, $\mu$ a constant satisfying \eqref{eqthmCE} and $(\mathbb{F}^m,\langle,\rangle)$ an $m$-dimensional Riemannian manifold with $m>1$ and Ricci tensor $Ric_{\Bbb{F}}=\mu\langle,\rangle$.  When $B$ is compact, we assume that $\lambda$ is positive.   Then, the warped product manifold $(B\times_u\Bbb{F}^m, g)$, with $g=g_B+ u^2\langle,\rangle$,  is Einstein with Ricci tensor $Ric=\lambda g$.
\end{corollary}

\section{Preliminaries and  Proofs of Theorems \ref{thmA} and \ref{thmConstEin} }

Firstly, let us list a lemma which is crucial for the proof of our result. Recall that the divergence of a $(1,1)$-tensor $T$ on a Riemannian manifold
$(M, g)$ is the $(0,1)$-tensor given by
\begin{equation}\label{divT}
(\mathrm{div} T)(u)(p)= \mathrm{tr}\left(w\mapsto (\nabla_wT)(u)(p)\right),
\end{equation}
where $p\in M$ and $u,w\in T_pM$. If $T$ is a $(0,2)$-tensor on $M$, one can associate with $T$ a unique $(1,1)$-tensor,
also denoted by $T$, according to
\begin{equation}
g(T(Z), Y)=T(Z,Y),
\end{equation}
for all $Y, Z \in \mathfrak{X}(M)$. Thus, we get
\begin{eqnarray}\label{FG}
\nonumber \mathrm{div}(\varphi T)&=&\varphi \mathrm{div} T+ T(\nabla\varphi,\cdot)\\
\nonumber \nabla(\varphi T)&=&\varphi\nabla T+d\varphi\otimes T\\
\frac{1}{2}d|\nabla\varphi|^2&=&\nabla^2\varphi(\nabla\varphi,\cdot)\\
\nonumber \mathrm{div}\nabla^2\varphi&=&Ric(\nabla\varphi,\cdot)+ d\Delta\varphi
\end{eqnarray}
for all $\varphi\in C^\infty(M).$ In particular, we have $\mathrm{div}(\varphi g)=d\varphi$ and the second contracted Bianchi identity:  $\frac{1}{2}d(\mathrm{tr} Ric)=\mathrm{div} Ric$.

In fact, the first three equations can be proved directly from  definitions, while the fourth one is a general fact already known in the literature, which can be obtained directly from \eqref{divT} and the definition of $\nabla^2\varphi$.

The next lemma enables  us to use the techniques by Barros and the first autor \cite{bjng}.
\begin{lemma}\label{lem1}
Let $T$ be a symmetric $(0,2)$-tensor on a  Riemanniana manifold $(M,g)$. Then we have
\be
\mathrm{div} (T(\varphi Z))= \varphi(\mathrm{div} T)(Z)+ \varphi\langle \nabla Z, T\rangle + T(\nabla\varphi ,Z),
\en
for each $Z\in\mathfrak{X}(M)$ and any smooth function  $\,\varphi$ on $M$.
\end{lemma}

Now we consider $(\Bbb{M}^n(c),g_0)$, a simply connected Riemannian manifold of constant sectional curvature $c\in\{-1,1\}.$
Let us denote by $\Bbb{R}^{n+1}_\nu$, $\nu\in\{0,1\}$, the vector space $\Bbb{R}^{n+1}$ endowed with the inner product $\langle\,,\,\rangle$ given by the standard way:
$$\langle\,v,w\,\rangle=\displaystyle\sum_{i=1}^{n}v_iw_i+(-1)^{\nu}v_{n+1}w_{n+1},$$
where $v=(v_1,\ldots,v_{n+1})$ and  $w=(w_1,\ldots,w_{n+1})$ are elements of $\Bbb{R}^{n+1}$. With this setting the standard sphere $(\Bbb{S}^n(1),g_0)$ is defined by
$$\Bbb{S}^n(1)=\{p\in \Bbb{R}^{n+1}_{0}; \langle p,p\rangle=1\}$$
while the standard hyperbolic space $(\Bbb{H}^n(-1),g_0)$ is given by
$$\Bbb{H}^n(-1)=\{p\in \Bbb{R}^{n+1}_{1}; \langle p,p\rangle=-1, p_{n+1}\geq1\},$$which is a spacelike hypersurface  of $\Bbb{R}^{n+1}_1$, i.e., the inner product $\langle\,,\,\rangle$  restricted to $\Bbb{H}^n(-1)$ is a Riemannian metric $g_0$. Following the same idea of \cite{br3} we have:

\begin{example}\label{ex2} Let $h_v$ be a height function with respect to a fixed unit vector $v\in\Bbb{R}^{n+1}$.
The quadruple $(\Bbb{M}^n(c),g_0,\nabla u,\frac{m}{u},\lambda),$ where $u=e^{\frac{f}{m}}$, $f=m\ln(\tau-c\frac{h_v}{n})$, $\tau$ is a real number such that $f$ is a nonconstant real function and $\la=c(n-1)+\frac{mc^2}{n\tau-ch_v}h_v,$ is a nontrivial structure of gradient $\frac{m}{u}$-almost Ricci soliton on $\Bbb{M}^n(c)$.
\end{example}
In fact, since $df=-\frac{mc}{n\tau-ch_v}dh_v$ and $\nabla^2h_v=-ch_v g_0$ we have
\begin{eqnarray*}
\nabla^2f = \nabla df &=& -\frac{mc}{n\tau-ch_v}\nabla dh_v - d\big(\frac{mc}{n\tau-ch_v}\big)\otimes dh_v\\
&=& -\frac{mc}{n\tau-ch_v}\nabla^2h_v - \frac{mc^2}{(n\tau-ch_v)^2}dh_v\otimes dh_v\\
&=& \frac{mc^2}{n\tau-ch_v}h_vg_0 - \frac{mc^2}{(n\tau-ch_v)^2}dh_v\otimes dh_v.
\end{eqnarray*}
Hence,
\be\no
\frac{m}{u}\nabla^2 u=\nabla^2f +\frac{1}{m}df\otimes df= \frac{mc^2}{n\tau-ch_v}h_vg_0.
\en
Since $Ric= c(n-1)g_0$, we have
\begin{equation*}
Ric + \frac{m}{u}\nabla^2 u = \big( c(n-1)+\frac{mc^2}{n\tau-ch_v}h_v \big)g_0.
\end{equation*}
As we had stated.

Analogously, consider $u=e^{\frac{f}{m}}$, where $f(x)=m\ln(\tau+|x|^2)$ and $\tau$ is a real number such that $f$ is a nonconstant real function and let $g_0$ be the canonical metric on $\mathbb{R}^n$. Then for $\la(x)=\frac{2m}{\tau+|x|^2}$, $(\Bbb{R}^n,g_0,\nabla u,\frac{m}{u},\la)$ is a nontrivial  gradient $h$-almost  Ricci soliton. On the other hand, the vector field
\begin{equation}\label{ex-nongrad}
X(x_1,\ldots, x_n)=\left(x_nx_1,x_nx_2,\ldots,x_nx_{n-1},\frac{x_n^2}{2}\right)
\end{equation}
is nonhomothetic and conformal and so defines  a nontrivial structure of $h$-almost Ricci soliton on $(\mathbb{R}^n,g_0)$, where $x_1, \ldots, x_n$ are the canonical coordinates in $\Bbb{R}^n$.

\begin{example}\label{ex-Einst}
 Suppose that $(\mathbb{F},\langle,\rangle)$ is an $(n-1)(n\geq 3)$-dimensional complete  Einstein manifold with $Ric_{\langle,\rangle}= -(n-2)l\langle,\rangle$, $l\geq 0$. Let $k$ be a negative constant  and define $f: \mathbb{R}\to\mathbb{R}^{+} $  by
\begin{equation}\label{f''}
f(t)=\frac{A}{\sqrt{-k}}\sinh(\sqrt{-k}t)+\sqrt{\frac {A^2+l}{-k}}\cosh(\sqrt{-k}t),
\end{equation}
where $A\neq 0$ is a constant.
Let $M^n=\mathbb{R}\times_f\Bbb{F}$ denote the $f$-warped product of  ${\mathbb R}$  and  $(\Bbb{F},\langle,\rangle)$. Namely, the $n$-dimensional, smooth product manifold  $M^n= {\mathbb R}\times\Bbb{F}$ is endowed with the metric
\begin{equation*}
g=dt\otimes dt + f(t)^2\langle,\rangle,
\end{equation*}
where $t$ is a global parameter of $\mathbb{R}$. It follows from Lemma 1.1 in \cite{prrs} that $(M^n,g)$ is Einstein with $Ric_g=(n-1)kg$.
Since $(M^n, g)$ is complete and $k<0$, there is a function $u$ on $(M^n,g)$ without critical points  satisfying $\nabla^2u+kug=0$ (see Theorem D in \cite{masahiko}). So, for each smooth function $h:\mathbb{R}\to\mathbb{R}$, if $\la=(n-1)k-hku$, then $(M^n,g,\nabla u,h,\la)$ is a  nontrivial structure of gradient $h$-almost  Ricci soliton on $(M^n,g)$.
\end{example}

Now we are ready to prove the main results of this paper.
\vskip0.2cm
{\it Proof of Theorem \ref{thmA}.}
For a symmetric $(0,2)$-tensor $T$ on  $(M^n,g)$, we  denote by $\overset\circ{T}$ the traceless tensor associated with  $T$, that is, $\overset\circ{T}=T-\frac{\mathrm{tr}(T)}n g$.
Let $Ric$ and $R$ be the Ricci tensor and the scalar curvature of $M^n$, respectively. Instead of using the constancy assumption on $R$, we shall prove our Theorem \ref{thmA} under the weaker condition that $\langle X, \nabla R\rangle\leq 0$ on $M^n$, if $h>0$, or $\langle X, \nabla R\rangle\geq 0$, if $h<0$. Setting $S=\frac 12 {\mathcal L}_X g$, we have from \eqref{eqfund1h} that
\be\label{Ric1}
\overset\circ{Ric} = -h\overset\circ{S},
\en
Taking $T=\overset\circ{Ric}, \,\varphi=1$ and $Z=X$ in Lemma \ref{lem1}, we obtain
\be\label{divRic1}
\mathrm{div}(\overset\circ{Ric}(X))=(\mathrm{div} \overset\circ{Ric})(X)+\langle\nabla X, \overset\circ{Ric}\rangle.
\en
It follows from the second contracted Bianch identity that
\be\label{divRic2}
(\mathrm{div} \overset\circ{Ric})(X)=\frac{n-2}{2n}\langle \nabla R, X\rangle.
\en
By a straightforward computation we infer
\be\label{intRic}
\langle\nabla X, \overset\circ{Ric}\rangle = \langle\overset\circ{Ric}, \overset\circ{S}\rangle=- h|\overset\circ{S}|^2.
\en
Combining \eqref{divRic1}-\eqref{intRic}, we have
\be\label{divRic3}
\mathrm{div}(\overset\circ{Ric}(X))=\frac{n-2}{2n}\langle \nabla R, X\rangle- h|\overset\circ{S}|^2.
\en
Integrating on $M^n$, we know that $\overset\circ{S}=0$. Hence $X$ is a nonhomothetic  conformal vector field and from \eqref{Ric1} $M^n$ is Einstein.
Let us set
\be\label{conformal}
{\mathcal L}_X g= 2 \rho g.
\en
where, by \eqref{eqfund1h}
\be\label{rho}
\rho =\frac{\mathrm{div} X}n=\frac 1h\left(\lambda -\frac Rn\right).
\en
Moreover, the conformal factor $\rho$ satisfies the following equation (see for example p.28 in \cite{yano}):
\be\label{eqThr}
\nabla^2\rho=-\frac{R}{n(n-1)}\rho g.
 \en
 Since $\rho$ is not constant, we conclude that $\frac{R}{n-1}$ is a nontrivial eigenvalue  of the Laplacian and therefore $R>0$.
 Consequently, $M^n$ is isometric to a sphere ${\mathbb S}^n(r)$, where $r=\sqrt{n(n-1)/R}$ is the radius of the sphere (Cf. \cite{obata}). It then follows  that $\rho$ is an eigenfunction corresponding to the first eigenvalue $\lambda_1=R/(n-1)$ of the sphere ${\mathbb S}^n(r)$.
Setting
\be\label{z}
u = - \frac {n(n-1)}R \rho,
\en
we obtain
\be
\frac 12 {\mathcal L}_{\nabla u} g = \nabla^2 u =-\frac{n(n-1)}{R}\nabla^2\rho= \rho g = \frac 12 {\mathcal L}_X g.
\en
This completes the proof of Theorem \ref{thmA}.
\qed
\vskip0.2cm
Before proving Theorem \ref{thmConstEin}, we shall need  the following lemma.
\begin{lemma}\label{pro princ}
For a gradient $(-\frac{m}{u})$-almost Ricci soliton $(M^n, g, \nabla u, -\frac{m}{u},\la)$, the following is valid
\begin{equation}\label{eqPprinc}
d\Big(\frac{n-2}{m}u^2\lambda - u\Delta u - (m-1)|\nabla u|^2\Big) - \frac{m+n-2}{m}\lambda du^2=0
\end{equation}
\end{lemma}
{\it Proof of Lemma \ref{pro princ}.}
In this case, we have
\begin{equation*}
Ric-\frac{m}{u}\nabla^2u-\lambda g=0\ \ \ \mbox{and}\ \ \ R=n\lambda+\frac{m}{u}\Delta u.
\end{equation*}
Using the formulas indicated in \eqref{FG} we get
\begin{eqnarray*}
0&=&\mathrm{div} Ric - \frac{m}{u}\mathrm{div}\nabla^2u + \nabla^2u\Big(\nabla\big(-\frac{m}{u}\big),\cdot\Big)-d\lambda\\
&=&\frac{1}{2}d\big(n\lambda+\frac{m}{u}\Delta u\big)-\frac{m}{u}\big(Ric(\nabla u,\cdot)+d\Delta u\big)+\frac{m}{u^2}\nabla^2u(\nabla u,\cdot)-d\lambda\\
&=&\frac{n-2}{2}d\lambda + \frac{1}{2}d(\frac{m}{u}\Delta u) -\frac{m}{u}d\Delta u -\frac{m}{u}\big(\lambda du +\frac{m}{u}\nabla^2u(\nabla u,\cdot)\big) +\frac{m}{2u^2}d|\nabla u|^2\\
&=&\frac{n-2}{2}d\lambda - \frac{m}{2u^2}(ud\Delta u+\Delta udu) -\frac{m}{u^2}\lambda udu -\frac{m^2}{2u^2}d|\nabla u|^2 +\frac{m}{2u^2}d|\nabla u|^2\\
&=&\frac{n-2}{2}d\lambda -\frac{m}{2u^2}\lambda du^2 - \frac{m}{2u^2}d(u\Delta u) -\frac{m^2}{2u^2}d|\nabla u|^2 +\frac{m}{2u^2}d|\nabla u|^2.
\end{eqnarray*}
Multiplying by $\frac{2u^2}{m}$ we obtain
\begin{eqnarray*}
\frac{n-2}{m}u^2d\lambda -\lambda du^2 - d(u\Delta u) - (m-1)d|\nabla u|^2=0,
\end{eqnarray*}
which is \eqref{eqPprinc}.
\qed
\vskip0.2cm
{\it Proof of Theorem \ref{thmConstEin}.}

$i)$ Making $\lambda$ a constant in \eqref{eqPprinc} we get \eqref{eqthmCE}.

$ii)$ Since $M$ is compact, we can take $p, q\in M$ such that
\begin{equation}\label{eqMm}
u(p)=\max_{x\in M} u(x), \ u(q)=\min_{x\in M} u(x).
\end{equation}
Then we have
\begin{equation}\label{consMm}
u(p)>0,\ u(q)>0,\ \nabla u(p)=\nabla u(q)=0,\ \Delta u(p)\leq 0\leq\Delta u(q).
\end{equation}
Let us consider firstly the case $\lambda=0$. In this case, we have from \eqref{eqthmCE} that
\begin{equation}\label{eqAuxMm}
 u\Delta u + (m-1)|\nabla u|^2 = \mu,
\end{equation}
which, combining with \eqref{consMm}, gives
\begin{equation*}
0\geq u(p)\Delta u(p)= \mu =u(q)\Delta u(q)\geq 0,
\end{equation*}
and so, $\mu=0.$ Consequently, we get from \eqref{eqAuxMm} that
\begin{equation}\label{eq=0}
 u\Delta u + (m-1)|\nabla u|^2 =0.
\end{equation}
Integrating  on $M$, one  obtains
\begin{equation}
(m-2)\int_{M} |\nabla u|^2=0.
\end{equation}
When $m\neq 2$, the above equation tells us that $u$ is a constant. On the other hand, when $m=2$,
we have from \eqref{eq=0} that
\begin{eqnarray*}
\frac{1}{2}\Delta u^2 = u\Delta u+|\nabla u|^2=0,
\end{eqnarray*}
which again implies that $u$ is a constant.

Consider now the case $\lambda <0$. Observe that
\begin{eqnarray*}
\lambda u^2(p)\geq \mu \geq \lambda u^2(q),
\end{eqnarray*}
which can be easily deduced from \eqref{eqthmCE} and \eqref{consMm}. Thus, we have  $u^2(p)\leq u^2(q)$. Consequently, $u$ is a constant.
\qed

\section{The case noncompact}

The proof of  Theorem \ref{thmA} shows that any condition that renders $\mathring{S}=0$ in the equation \eqref{divRic3} will entail that the manifold is  Einstein and the vector field $X$ is conformal, with conformal factor satisfying \eqref{eqThr}. This allows us to  get the following result.

\begin{theorem}\label{thmNoncompact}
Let $(M^n, g, X, h, \la)$ be a nontrivial noncompact $h$-almost Ricci soliton, with $h$ having defined signal. Suppose that $\mathcal{L}_{X}R \leq 0$  and $|\mathring{Ric}(X)|$ lies in $\emph{L}^{1}(M^n).$ Then $(M^n,g)$ is a Einstein manifold with non positive scalar curvature $R.$ Moreover:
\begin{enumerate}
\item\label{2_thmB} If $R=0$, then $(M^n, g)$  is isometric to the Euclidean space $(\mathbb{R}^n, g_0)$.
\item\label{1_thmB} If $R<0$, then $\mathcal{L}_Xg=\mathcal{L}_{\nabla u}g$ with potential function $u$ given by \eqref{z} and $M^n$  is isometric to a  hyperbolic space provided that $u$ has only one critical point, or a pseudo-hyperbolic space provided that $u$ has no critical point.
\end{enumerate}

\end{theorem}

\begin{proof}
Suppose that $\mathcal{L}_{X}R\leq0$. Then, the equation \eqref{divRic3} gives that $\mathrm{div}\big(\mathring{Ric}(X)\big)\leq0.$ Since $\big|\mathring{Ric}(\nabla u|\big)$ lies in $\emph{L}^{1}(M^n)$ we can use Proposition $1$ of \cite{caminha} to deduce that $\mathrm{div}\big(\mathring{Ric}(X)\big)=0.$ Thus, $M^n$ is Einstein and $X$ is a nonhomothetic conformal vector field. Moreover, we can assume that the equations \eqref{conformal}-\eqref{eqThr} hold. Furthermore, we have $R\leq0$, since $M^n$ is noncompact. Moreover:

1) If $R=0$, we conclude from item (ii) of   Theorem G in \cite{masahiko} that $(M^n, g)$ is isometric to the $n$-dimensional Euclidean space $(\mathbb{R}^n, g_0)$.

2) If $R<0$, we can replace $X$ by $\nabla u$, where $u$ is given by \eqref{z}. In particular, we obtain a complete classification by using Theorem 2 of \cite{Tashiro} or Theorem G of \cite{masahiko} and analyzing the critical point of the potential function $u$. In short, for $R=n(n-1)k$, $(M^n,g)$  is isometric to the hyperbolic space $(\Bbb{H}^n,-(1/k)g_0)=\Bbb{R}\times_{f}\Bbb{R}^{n-1},$ $f(t)=e^{\pm\sqrt{-k}t}$, provided that the potential function $u$ has only one critical point, or a pseudo-hyperbolic space, that is, a warped product $\Bbb{R}\times_{f}\Bbb{F}$, where $f$ is a solution of  $f''+kf=0$, and $\Bbb{F}$ is a complete Einstein manifold, provided that $u$ has no critical point. Finally, we mention that any of the above two cases can  occur. The first one is  contained in Example \ref{ex2} and the second one is Example \ref{ex-Einst}.
\end{proof}

\end{document}